\documentclass[conference]{IEEEconf}
\IEEEoverridecommandlockouts
\usepackage[noadjust]{cite}
\usepackage{amsmath,amssymb,amsfonts}
\usepackage{algorithmic}
\usepackage{graphicx}
\usepackage{textcomp}
\usepackage{xcolor}
\usepackage{dsfont}
\usepackage{amsthm}

\let \labelindent \relax
\usepackage{enumitem}
\usepackage{subcaption}
\usepackage[font=small]{caption}

\usepackage{accents}

\newtheorem{theorem}{Theorem}[section]

\newtheorem{lemma}[theorem]{Lemma}

\newtheorem{proposition}[theorem]{Proposition}
\newtheorem{assumption}{Assumption}

\newcommand{\bs}[1]{\boldsymbol{#1}}
\newcommand{\R}{{\mathds{R}}}

\newcommand{\Na}{{N_{\rm a}}}
\newcommand{\No}{{N_{\rm o}}}

\newcommand{\Zz}{\bs{Z}}

\newcommand{\edits}[1]{{\color{black} #1} }

\def\BibTeX{{\rm B\kern-.05em{\sc i\kern-.025em b}\kern-.08em
    T\kern-.1667em\lower.7ex\hbox{E}\kern-.125emX}}
\begin{document}

\title{Sustained oscillations in multi-topic belief dynamics over signed networks
\thanks{Supported by ONR grant N00014-19-1-2556, ARO grant W911NF-18-1-0325, DGAPA-UNAM PAPIIT grant IN102420, and Conacyt grant A1-S-10610, and by  
NSF Graduate Research Fellowship  
DGE-2039656.}
}

\author{Anastasia~Bizyaeva,
        Alessio~Franci,
        and~Naomi~Ehrich~Leonard
\thanks{A. Bizyaeva and N.E. Leonard are with the Dept.
of Mechanical and Aerospace Engineering, Princeton University, Princeton,
NJ, 08544 USA;  bizyaeva@princeton.edu, naomi@princeton.edu.}
\thanks{A. Franci is with the Dept. of Mathematics, National Autonomous University of Mexico, 04510 Mexico City, Mexico; afranci@ciencias.unam.mx}}%

\maketitle

\begin{abstract}
We study the dynamics of belief formation on multiple interconnected topics in networks of agents with a shared belief system. We establish sufficient conditions and necessary conditions under which sustained oscillations of beliefs arise on the network in a Hopf bifurcation and characterize the role of the communication graph and the belief system graph in shaping the relative phase and amplitude patterns of the oscillations. Additionally, we distinguish broad classes of graphs that exhibit such oscillations from those that do not. 
\end{abstract}



\section{Introduction}



\edits{Having the means to evaluate what can happen when a group of social agents forms beliefs on a set of related topics is key to understanding belief propagation in human social networks and to enabling decentralized decision-making  in teams of robots and other distributed technological systems. Dynamic models of social belief formation provide a tool for systematic investigation of  belief  processes and for principled design of distributed algorithms for decision-making.}

\edits{In this paper we investigate conditions under which oscillations emerge in the beliefs of agents in social networks. Temporal oscillations in attitudes and beliefs may be an important feature of individual cognition  \cite{fink2002oscillation}. Oscillations in beliefs are common in social systems; e.g., periodic swings in public opinion between more  conservative and more liberal attitudes are characteristic of the American electorate \cite{stimson2018public}. In a multi-robot problem such as task allocation, it may be important to reliably promote or  avoid oscillations.
Designing oscillations will also be necessary for building electronic circuits with complicated, but well controlled, oscillation patterns as those needed for neuromorphic applications~\cite{Liu2014event}.}

\edits{However, sustained oscillations are rarely observed in popular models of belief formation.} Classically, formation of beliefs or opinions on a single topic is modeled as a discrete-time or continuous-time linear weighted averaging process on a network \cite{DeGroot1974,OlfatiSaber2004}. \edits{For the multi-topic scenario,}  multi-dimensional averaging models have been investigated, e.g., see \cite{friedkin2016network,parsegov2016novel,ye2019consensus,pan2018bipartite,ye2020continuous,ahn2020opinion,wang2022characterizing}. \edits{According to linear models, the beliefs of agents in a static social network typically converge to an equilibrium. The study of these models is thus concerned with characterizing the agents' beliefs at steady state. }


Recently, an alternative modeling paradigm for social opinion formation was proposed that assumes the belief or opinion update rules of agents to be \textit{nonlinear} \cite{bizyaeva2023tac,FranciSIADS}. The nonlinearity is deceptively simple: each agent saturates information it accumulates from its social network. \edits{The imposition of a saturating function is a well-motivated and mild extension of classic averaging models \cite{bizyaeva2023tac,FranciSIADS}. Despite the simplicity, networked beliefs that follow this nonlinear update rule can have}  dramatically different \edits{properties} from those predicted by classic averaging models, \edits{ including sustained oscillations.} \edits{These nonlinear dynamics are also general. Beyond opinion formation, they are closely related to recurrent neural network and neuromorphic electronic circuit models.} To date, analysis of these nonlinear dynamics focused on characterizing multi-stable equilibria 
\cite{bizyaeva2023tac,FranciSIADS,bizyaeva2022switching}. \edits{In this paper we add to this body of work and present novel analysis that characterizes the emergence of belief oscillations and their properties as a function of design parameters including mixed-sign network structure.}

Our main contributions are as follows. 1) We establish sufficient conditions and necessary conditions for the onset of stable sustained oscillations in belief dynamics. 2) We characterize the relative phase and amplitude patterns of the oscillations in terms of the parameters of the model. 

Section~\ref{sec: preliminaries} reviews mathematical preliminaries. Section~\ref{sec: model} introduces the belief dynamics model. Section~\ref{sec: indecision breaking} presents the main results. Classes of graphs that can  lead to oscillations are distinguished from those that cannot in 
Section~\ref{sec:cond}. Section~\ref{sec:ex} presents numerical examples.

\section{Mathematical preliminaries}
\label{sec: preliminaries}

\subsection{Notation}

For  $x = a + i b = r e^{i \phi} \in \mathds{C}$, $\overline{x} = a - i b = r e^{-i \phi}$ is its complex conjugate, $|x| = \sqrt{x \overline{x}} = r$  its modulus, and $\operatorname{arg}(x)$  its argument $\phi$.  
The inner product of  vectors $\mathbf{v},\mathbf{w}$ is $\langle \mathbf{w},\mathbf{v} \rangle = \overline{\mathbf{w}}^T \mathbf{v}$. $\mathbf{0} \in \mathds{R}^N$ is the zero vector and $\operatorname{diag}(\mathbf{v})$ is the diagonal matrix with diagonal entries the elements of $\mathbf{v}$.

The spectrum of $A \in \mathds{R}^{n\times n}$ is $\sigma(A) = \{\lambda_1, \dots, \lambda_{n} \}$ and its spectral radius $\rho(A)= \operatorname{max}\{|\lambda_i|, \ \lambda_i \in \sigma(A)\}$. The kernel of $A$ is $\mathcal{N}(A) = \{\mathbf{v} \in \mathds{R}^n \ s.t. \ A \mathbf{v} = \mathbf{0} \}$. An eigenvalue $\lambda \in \sigma(A)$ is a leading eigenvalue of $A$ if $\operatorname{Re}(\lambda) \geq \operatorname{Re}(\mu)$ for all $\mu \in \sigma(A)$. A leading eigenvalue $\lambda$ of $A$ is a dominant eigenvalue if $\lambda = \rho(A)$.
Given  vectors $\mathbf{v},\mathbf{w}$ or  matrices $M,N$, we say 
$\mathbf{v} \succ \mathbf{w}$ if $v_i > w_i$ for all $i$ and $M \succ N$ if $M_{ij} > N_{ij}$ for all $i,j$. For  matrices $M,N\in\mathds{R}^{m \times n}$, the element-wise Hadamard product $M \odot N \in \mathds{R}^{m \times n}$ is defined as $(M \odot N)_{ij} = M_{ij} N_{ij}$. 
For  matrices $A  = (a_{ij}) \in \mathds{R}^{m \times n}, B= (b_{ij}) \in \mathds{R}^{l\times k}$ the Kronecker product $A\otimes B  \in \mathds{R}^{m l \times nk}$ is defined as
\begin{equation*}
    A \otimes B = \begin{pmatrix} a_{11} B & \dots & a_{1n} B \\ \vdots & \ddots & \vdots \\ a_{m1} B & \dots & a_{mn} B  \end{pmatrix}.
\end{equation*}
A real square matrix $A$ has the strong Perron-Frobenius property if it has a unique dominant eigenvalue $\lambda = \rho(A)$ satisfying $\lambda \geq |\lambda_i|$ for all $\lambda_i \neq \lambda$ in $\sigma(A)$ and its corresponding eigenvector satisfies $\mathbf{v}\succ \mathbf{0}$.
$A$  is irreducible if it cannot be transformed into an upper triangular matrix through similarity transformations. $A$ is eventually positive (eventually nonnegative) if there exists a positive integer $k_0$ such that $A^k \succ 0_{N \times N}$ ($A^k \succeq 0_{N \times N}$) for all integers $k > k_0$.
\begin{proposition} \cite[Theorem 2.2]{noutsos2006perron}
The following statements are equivalent for a real square matrix $A$: (1) $A$ and $A^T$ have the strong Perron-Frobenius property; (2) $A$ is eventually positive; (3) $A^T$ is eventually positive. \label{prop:PerFr}
\end{proposition}

\subsection{Signed graphs}

A graph $\mathcal{G} = (\mathcal{V},\mathcal{E})$ is  a set of nodes $\mathcal{V} = \{1, \dots, N \}$ and a set of edges $\mathcal{E}$.
We assume the graph is simple, i.e., there is at most one edge between any two nodes.
The \textit{adjacency matrix} $A = (a_{ik})$ of  $\mathcal G$ satisfies $a_{ik} = 0$ if $e_{ik} \not \in \mathcal{E}$ and $a_{ik} \neq 0$ otherwise. $A$ is \textit{weighted} if nonzero entries $a_{ik}\in\mathbb R$.
$\mathcal G$ is \textit{unweighted} if $a_{ik} \in \{0,1\}$ or \textit{signed unweighted} if $a_{ik} \in \{0,1,-1\}$ for all $i,k \in \mathcal{V}$. $\mathcal G$ is \textit{undirected} whenever $a_{ik} = a_{ki}$ for all $i,k \in \mathcal{V}$, and $\textit{directed}$ otherwise.

The in-degree of node $i$ on $\mathcal{G}$ is 
$\sum_{k} a_{ik}$.
A path on  $\mathcal{G}$ is a finite or infinite sequence of edges that joins a sequence of nodes.
$\mathcal{G}$ is  \textit{strongly connected} if there exists a path from any node to any other node. $\mathcal{G}$ is strongly connected if and only if $A$ is an irreducible matrix. A \textit{switching matrix} $M$ is a diagonal matrix with diagonal entries that are either $1$ or $-1$. Two graphs $\mathcal{G}_1$, $\mathcal{G}_2$ with adjacency matrices $A_1,A_2$ are \textit{switching equivalent} whenever $A_1 = M A_2 M$.  

\subsection{Hopf bifurcation}
\label{sec: hopf bif}

Assume without loss of generality that $(\mathbf{x},p) = (\mathbf{0},0)$ is an equilibrium of a  system $\dot{\mathbf{x}} = \mathbf{f}(\mathbf{x},p)$, where $\mathbf{x}$ is the  state and $p$  a parameter. Then $(\mathbf{0},0)$ is a \textit{Hopf bifurcation point} if it satisfies the following:
{\it i)} The Jacobian  $D\mathbf{f}(\mathbf{0},0)$ has a complex conjugate pair of eigenvalues $\pm i \omega(0)$;
{\it ii)} No other eigenvalues of $D\mathbf{f}(\mathbf{0},0)$ lie on the imaginary axis; {\it iii)} Let $\lambda(p) = r(p) + i \omega(p)$, $\overline{\lambda}(p) = r(p) - i \omega(p)$ be the eigenvalues of $D\mathbf{f}(\mathbf{x},p)$ that are smoothly parametrized by $p$ and for which $r(0) = 0$; then $\frac{\partial r}{\partial p}{(\mathbf{0},0)} \neq 0$.
We  use Lyapunov-Schmidt reduction methods \cite[Chapter VIII]{Golubitsky1985} to study 
the 
limit cycles that emerge through a Hopf bifurcation.

\section{Belief formation model}
\label{sec: model}

We study a nonlinear model of $N_a$ homogeneous agents forming beliefs about $N_o$ topics, adapted from \cite{bizyaeva2023tac,FranciSIADS}. $z_{ij} \in \mathds{R}$ is the belief of agent $i$ about topic $j$. Whenever $z_{ij} > 0 (< 0)$, agent $i$ is in favor of (in opposition to) topic $j$, and when $z_{ij} = 0$ it has a neutral belief on the topic. The magnitude $|z_{ij}|$ signifies the strength of commitment to the belief on topic $j$. The total belief state of agent $i$ is the vector $\Zz_i = (z_{i1},\dots, z_{i N_o}) \in \mathds{R}^{N_o}$, and the total network belief state is $\Zz = (\Zz_1, \dots, \Zz_{N_a}) \in \mathds{R}^{N_a N_o}$. We say agents $i$ and $k$ agree (disagree) on topic $j$ if they both form a non-neutral belief on the topic and $\operatorname{sign}(z_{ij}) = \operatorname{sign}(z_{kj})  (\neq \operatorname{sign}(z_{kj}) )$. Agent $i$ updates its belief on topic $j$ in continuous time as
\begin{multline}
   \dot{z}_{ij}=-d \ z_{ij} + 
	u\left(  S_{1}\!\!\left( \alpha z_{ij} + \gamma  \textstyle\sum_{\substack{k=1 \\ k \neq i}}^\Na(A_a)_{ik} z_{kj}\right) \right.  \\
	\left. + \textstyle\sum_{\substack{l\neq j\\l=1}}^\No S_{2} \!\!\left( \beta (A_o)_{jl} z_{il} + \delta (A_o)_{jl}\textstyle\sum_{\substack{k = 1\\k \neq i}}^\Na (A_a)_{ik} z_{kl}\right)\right) \label{EQ:value_dynamics}
\end{multline}
where $S_1, S_2: \mathds{R} \to \mathds{R}$ are bounded saturation functions satisfying $S_r(0) = 0$, $S'_r(0) = 1$, $S_r''(0) = 0$, $S_r'''(0) \neq 0$, with an odd symmetry $S_r(-y) = -S_r(y)$ where $r \in \{1,2\}$. 
$S_1$ saturates same-topic information and $S_2$ saturates inter-topic information. \edits{These saturations reflect that social network influence on each topic is bounded, and that an agent is maximally affected by small changes in its neighbors' beliefs on a topic when their weighted average is close to zero.} Parameter $d > 0$  represents the agents' resistance to forming strong beliefs. Parameter $u \geq 0$ regulates the amount of \textit{attention} agents allocate towards their social interactions, or their  \textit{susceptibility} to social influence.  

There are two signed directed graphs underlying the belief formation process. One is the \textit{communication} graph $\mathcal{G}_a = (\mathcal{V}_a, \mathcal{E}_a, s_a)$, with  signed adjacency matrix $A_a \in \mathds{R}^{N_a \times N_a}$.  $(A_a)_{ik} = 1$ means agent $i$ is \textit{cooperative} towards agent $k$, and  $(A_a)_{ik} = -1$ means it is \textit{antagonistic} towards agent $k$. The other is the \textit{belief system} graph $\mathcal{G}_o = (\mathcal{V}_o, \mathcal{E}_o, s_o)$, with  signed adjacency matrix $A_o \in \mathds{R}^{N_o \times N_o}$. The graph $\mathcal{G}_o$ encodes the logical interdependence between different topics in the set $\mathcal{V}_o$.  Whenever $(A_o)_{jl} = 1$($-1$), topic $j$ is positively (negatively) aligned with topic $l$ according to the belief system, and whenever $(A_o)_{jl} = 0$, topic $j$ is independent of topic $l$. In the model \eqref{EQ:value_dynamics} we assume that all agents form beliefs following a single shared belief system.

The gains $\alpha, \gamma, \beta, \delta \geq 0$ regulate the relative strengths of influence on beliefs in \eqref{EQ:value_dynamics}. 
$\alpha$ is the strength of agent self-reinforcement of already-held beliefs; $\beta$ is the strength of agent internal adherence to the belief system  $\mathcal{G}_o$.  $\gamma$ is the strength of agent social imitation, i.e. to mimic the beliefs of  neighbors towards whom the agent is cooperative and to oppose the beliefs of those towards whom it is antagonistic. $\delta$ is  agent ideological commitment; when $\delta$ is large,  agents evaluate their neighbors' influence more holistically according to the belief system $\mathcal{G}_o$ rather than through pure imitation along each topic. 
An illustration of these four effects and their respective cumulative weights in the model \eqref{EQ:value_dynamics} is shown in Fig. \ref{fig:comm-gains}.

\begin{figure}
    \centering
    \includegraphics[width=0.9\columnwidth]{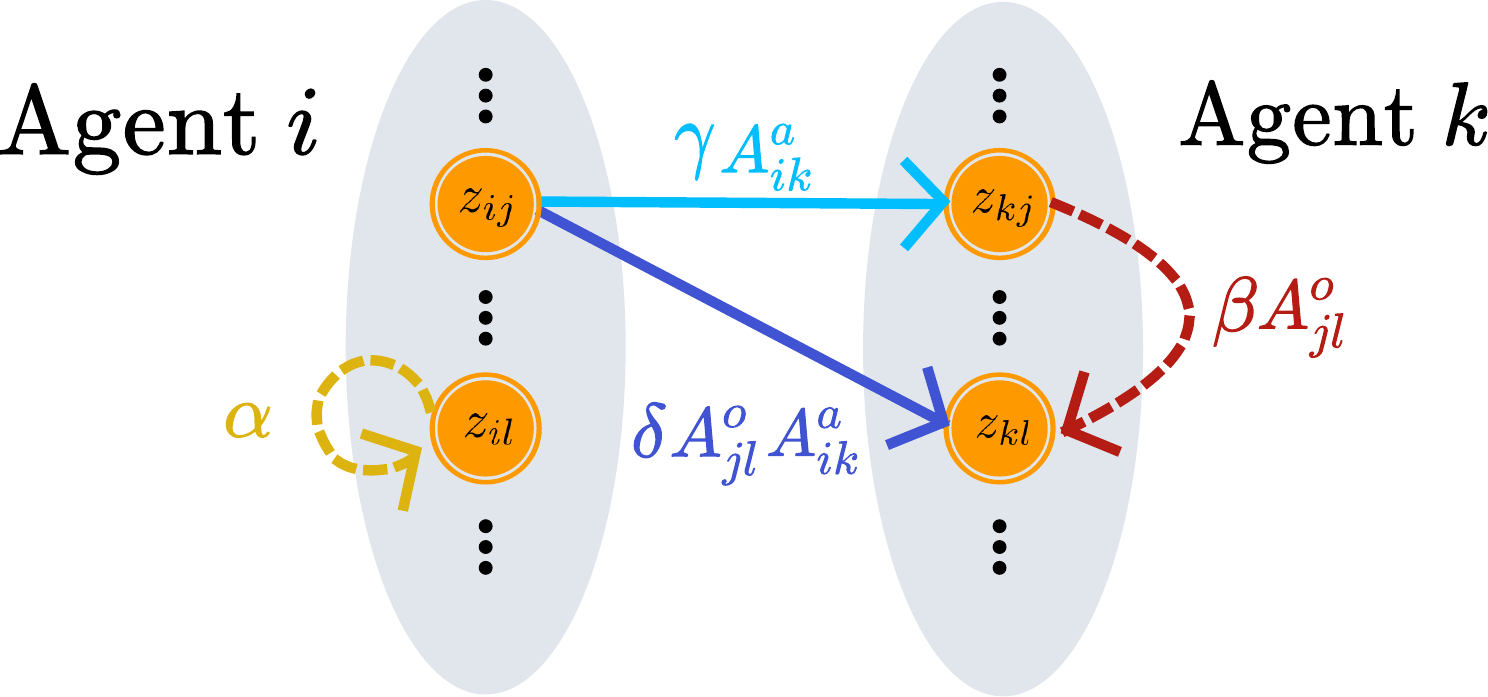}
    \caption{Four classes of communication weights in the model \eqref{EQ:value_dynamics}. Arrow direction follows sensing convention.}
    \label{fig:comm-gains}
\end{figure}


\section{Indecision-breaking and oscillations}
\label{sec: indecision breaking}

We establish sufficient conditions for the onset of small-amplitude periodic oscillations in the dynamics of beliefs  \eqref{EQ:value_dynamics}. The indecision state $\Zz = \mathbf{0}$ in which all agents have neutral beliefs on all topics is an equilibrium of \eqref{EQ:value_dynamics} for all parameter values. To establish the onset of oscillations we first study the stability of $\Zz = \mathbf{0}$. The Jacobian of  \eqref{EQ:value_dynamics} about $\Zz = \mathbf{0}$ is 
\begin{multline}
    J (\mathbf{0},u) = (- d + u \alpha) \mathcal{I}_{N_a}\otimes \mathcal{I}_{N_o} + u \gamma A_a \otimes \mathcal{I}_{N_o} \\
    + u \beta \mathcal{I}_{N_a} \otimes A_o + u \delta A_a \otimes A_o. \label{eq:jac}
\end{multline}
The following proposition connects the eigenvalues and eigenvectors of \eqref{eq:jac} to the eigenvalues and eigenvectors of the signed adjacency matrices $A_a$ and $A_o$.

\begin{proposition}[Eigenvalues and eigenvectors] \label{prop:eigen}
The following statements hold for \eqref{eq:jac}, with some selection of parameters $d,u,\alpha,\gamma,\beta,\delta$:

1) For each $\eta \in \sigma\big(J(\mathbf{0},u)\big)$, there exists $\lambda \in \sigma(A_a)$ and $\mu \in  \sigma(A_o)$ so that 
\begin{equation} \eta = -d + u (\alpha +  \gamma \lambda + \beta \mu + \delta \lambda \mu):= \eta(u,\lambda,\mu); \label{eq:jac_eig}
\end{equation}

2) Suppose $\lambda_i$ is an eigenvalue of $A_a$ with an eigenvector $\mathbf{v}_{a,i}$ and $\mu_j$ is an eigenvalue of $A_o$ with an eigenvector $\mathbf{v}_{o,j}$, then the vector $\mathbf{v}_{a,i} \otimes \mathbf{v}_{o,j}$ is an eigenvector of \eqref{eq:jac} with corresponding eigenvalue $\eta(u,\lambda_i,\mu_j)$.
\end{proposition}
\begin{proof}
1) By Schur's unitary triangularization theorem \cite[Theorem 2.3.1]{horn2012matrix} there exist unitary matrices $U \in \mathbb{C}_{N_a \times N_a}$, $V \in \mathbb{C}_{N_o \times N_o}$ such that $U^* A_a U = \Delta_a$, $V^* A_o V = \Delta_o$ where $\Delta_a,\Delta_o$ are upper triangular complex matrices with eigenvalues of $A_a, A_o$ on the diagonal. Then, using the mixed-product property of the Kronecker product,
$
\Delta_J = (U \otimes V)^* J(\mathbf{0},u) (U \otimes V) 
    = (- d + u \alpha) (U^* \mathcal{I}_{N_a} U)\otimes (V^* \mathcal{I}_{N_o} V)
    + u \gamma (U^* A_a U) \otimes (V^* \mathcal{I}_{N_o} V) + u \beta (U^* \mathcal{I}_{N_a} U) \otimes (V^* A_o V)
    + u \delta (U^* A_a U)\otimes (V^* A_o V)
    = (- d + u \alpha) \mathcal{I}_{N_a}\otimes \mathcal{I}_{N_o} + u \gamma \Delta_a \otimes \mathcal{I}_{N_o} + u \beta \mathcal{I}_{N_a} \otimes \Delta_o  + u \delta \Delta_a \otimes \Delta_o
$.
The matrix $\Delta_J$ is upper triangular, with its diagonal entries corresponding to the eigenvalues of $J(\mathbf{0}_N,u)$. By inspection we see that all diagonal entries of $\Delta_J$ have the form $\eta(u,\lambda,\mu)$ for some $\lambda \in \sigma(A_a)$, $\mu \in \sigma(A_o)$.

2)  $\mathbf{v}_{a,i} \otimes \mathbf{v}_{o,j}$ is an eigenvector of $A_a \otimes A_o$ with eigenvalue $\lambda_i \mu_i$ by \cite[Theorem 4.2.12]{horn1991topics}; it is also an eigenvector of $ \mathcal{I}_{N_a} \otimes \mathcal{I}_{N_o}$, $ A_a \otimes \mathcal{I}_{N_o} $, and $ \mathcal{I}_{N_a} \otimes A_o $ with corresponding eigenvalues $1, \lambda_i, \mu_i$, respectively, and the proposition statement follows from multiplying \eqref{eq:jac} by $\mathbf{v}_{a,i} \otimes \mathbf{v}_{o,j}$. 
\end{proof}
Whenever eigenvalues $\lambda\in \sigma(A_a)$ and $\mu \in \sigma(A_o)$ define an eigenvalue $\eta \in \sigma\big(J(\mathbf{0},u)\big)$ through  \eqref{eq:jac_eig}, we  say  $\lambda$ and $\mu$ \textit{generate} $\eta$. We define the maximal real part of the social network contribution to the eigenvalue \eqref{eq:jac_eig} of the Jacobian  as 
\begin{equation}
    K = \operatorname{max}_{\lambda \in \sigma(A_a), \mu \in \sigma(A_o)} \operatorname{Re}\left(\alpha + \gamma \lambda + \beta \mu + \delta \lambda \mu \right). \label{eq:K}
\end{equation}
For every leading eigenvalue $\eta_{max}$ of \eqref{eq:jac}, $\operatorname{Re}(\eta_{max}) = - d + u K$. In the following lemma we establish existence of a critical value of attention to social interactions at which the neutral equilibrium $\Zz = \mathbf{0}$ loses stability.

\begin{lemma}[Stability of origin] \label{lem:stab} Consider \eqref{EQ:value_dynamics} and suppose $K > 0$. If $u < u^* := d/K$, the neutral equilibrium $\Zz = \mathbf{0}$ is locally exponentially stable.  If $u > u^*$ it is unstable.
\end{lemma}
\begin{proof}
Let $\eta_{max}$ be a leading eigenvalue of \eqref{eq:jac}. Whenever $u < u^* (> u^*)$,  $\operatorname{Re}(\eta_{max}) < 0 (> 0)$. Furthermore since $\eta_{max}$ is a leading eigenvalue, whenever $u < u^*$, $\operatorname{Re}(\eta) < 0$ for all $\eta \in \sigma(J(\mathbf{0},u))$. The stability conclusions follow by Lyapunov's indirect method \cite[Theorem 4.7]{Khalil2002}.
\end{proof}

Lemma \ref{lem:stab} establishes the existence of a \textit{bifurcation point} $u = u^*$ at which the origin loses stability. As a consequence of the center manifold theorem \cite[Theorem 3.2.1]{guckenheimer2013nonlinear} for values of attention parameter $u$ in a neighborhood of $u^*$, trajectories of \eqref{EQ:value_dynamics} that start in a neighborhood of $\Zz = \mathbf{0}$ will settle on an attracting manifold with dimension determined by the number of eigenvalues of $J(\mathbf{0},u^*)$ with zero real part. We examine the onset of network oscillations along this manifold which result from a \textit{Hopf bifurcation}.


The following standing assumption ensures that Conditions {\it i)}, {\it ii)} for a Hopf bifurcation (Section~\ref{sec: hopf bif}) are satisfied.

\begin{assumption}
The leading eigenvalues of \eqref{eq:jac} are a complex-conjugate pair, $\eta_{+}$ and $\eta_{-} = \overline{\eta}_{+}$, with $\operatorname{Im}(\eta_{\pm}) \neq 0$. \label{ass1}
\end{assumption}

\edits{In Section \ref{sec:cond} we establish several broad classes of graphs for which this assumption is either always, or never, satisfied. }Let the pair $\lambda^\dagger = \lambda_a + i \lambda_c \in \sigma(A_a)$ and $\mu^\dagger = \mu_o + i \mu_c \in \sigma(A_o)$ generate one of the leading eigenvalues of Assumption \ref{ass1} according to \eqref{eq:jac_eig}, i.e. suppose that either $\eta_+ = \eta(u,\lambda^\dagger,\mu^{\dagger})$ or $\eta_- = \eta(u,\lambda^{\dagger},\mu^{\dagger})$. Then it holds that
\begin{subequations} \label{eq:eigs-form}
\begin{equation}
    \operatorname{Re}(\eta_{\pm}) = - d + u \big(\alpha + \gamma \lambda_a + \beta \mu_o + \delta (\lambda_a \mu_o - \lambda_c \mu_c) \big), \label{eq:eigs-form-a}
\end{equation}
\begin{equation}
    \operatorname{Im}(\eta_{\pm}) = \pm  u \big| \gamma \lambda_c + \beta \mu_c + \delta ( \lambda_a \mu_c + \lambda_c \mu_o )\big|.\label{eq:eigs-form-b}
\end{equation}
\end{subequations}
Note that in this case, 
\begin{equation}
    K = \alpha + \gamma \lambda_a + \beta \mu_o + \delta (\lambda_a \mu_o - \lambda_c \mu_c). \label{eq:K-osc}
\end{equation} We are now ready to establish our first main result.

\begin{theorem}[Hopf bifurcation] \label{thm:hopf}
Consider \eqref{EQ:value_dynamics} with communication graph $\mathcal{G}_a$ and belief system graph $\mathcal{G}_o$. Let Assumption \ref{ass1} hold, and suppose $K > 0$. 

Suppose $\lambda^{\dagger}\in\sigma(A_a), \mu^{\dagger} \in \sigma(A_o)$ generate $\eta_{+}(u)$. Let $\mathbf{w}_a,\mathbf{v}_a \in \mathds{C}^{N_a}$ be the left and right eigenvectors of $A_a$ corresponding to $\lambda^{\dagger}$ and $\overline{\lambda^{\dagger}}$, respectively; let $\mathbf{w}_o,\mathbf{v}_o\in \mathbb{C}^{N_o}$ be the left and right eigenvectors of $A_o$ corresponding to $\mu^{\dagger}$ and $\overline{\mu^{\dagger}}$, respectively. Choose the eigenvectors to satisfy the biorthogonal normalization condition 
\begin{equation*}
    \langle \mathbf{w}_a \otimes \mathbf{w}_o, \mathbf{v}_a \otimes \mathbf{v}_o, \rangle = 2, \ \ \langle \overline{\mathbf{w}_a \otimes \mathbf{w}_o}, \mathbf{v}_a \otimes \mathbf{v}_o, \rangle = 0. \label{eq:eigenvector-normalization-hopf}
 \end{equation*}

1) There is a unique 3-dimensional center manifold $W^c \subset \R^{N_a N_o} \times \R{}$ passing through $(\Zz,u) = (\mathbf{0},u^*)$, tangent to $\operatorname{span}\{ \operatorname{Re}(\mathbf{v}_a \otimes \mathbf{v}_o), \operatorname{Im}(\mathbf{v}_a \otimes \mathbf{v}_o) \} $ at $u = u^* = d/K$. There is a family of periodic orbits of \eqref{EQ:value_dynamics} that bifurcates from the neutral equilibrium $\Zz = \mathbf{0} $ along $W_c$ at $u = u^*$;

2) Let $b = \operatorname{Re}\Bigg( \Big( S_1'''(0) \left(\alpha + \gamma \lambda^{\dagger}\right) \left| \alpha + \gamma \lambda^{\dagger}\right|^2 +$
\begin{gather}
+ S_2'''(0) \left(\beta + \delta  \lambda^{\dagger}\right)\mu^{\dagger}\left| \beta +  \delta \lambda^{\dagger}\right|^2 \left| \mu^\dagger \right|^2 \Big)  \label{eq:Hopf_stab_cond}  \\
\times \langle \mathbf{w}_a\otimes \mathbf{w}_o , \lvert\mathbf{v}_a\otimes \mathbf{v}_o\rvert^2 \odot (\mathbf{v}_a\otimes \mathbf{v}_o) \rangle \Bigg)\nonumber
\end{gather}
where $\lvert\mathbf{x}\rvert^2 = \overline{\mathbf{x}} \odot \mathbf{x}$. Whenever $b < 0$ the bifurcating periodic solutions appear supercritically (for $u > u^*$) and are locally asymptotically stable; whenever $b > 0$, the solutions appear subcritically (for $u < u^*$) and are unstable;

3) When $\lvert u - u^* \rvert$ is small, the period of the solutions is near $2 \pi /(u^*|\gamma \lambda_c + \beta \mu_c + \delta (\lambda_a \mu_c + \lambda_c \mu_o)|)$, the difference in phase between $z_{ij}(t)$ and $z_{kl}(t)$ is near $ \varphi = \operatorname{arg}((\mathbf{v}_a)_i (\mathbf{v}_{o})_j) - \operatorname{arg}((\mathbf{v}_a)_k (\mathbf{v}_{o})_l)$, and the amplitude of $z_{ij}(t)$ is greater than the amplitude of $z_{kl}(t)$ if and only if $\lvert (\mathbf{v}_a)_i \rvert \lvert (\mathbf{v}_o)_j\rvert > \lvert (\mathbf{v}_a)_k \rvert \lvert (\mathbf{v}_o)_l\rvert $.
\end{theorem}

Theorem~\ref{thm:hopf} provide sufficient conditions for the emergence of {\it stable} sustained oscillations at an indecision-breaking bifurcation. Statement 3) of Theorem~\ref{thm:hopf} relates the relative phase and amplitude pattern along the emerging oscillation to the spectral properties of $\mathcal{G}_a$ and $\mathcal{G}_o$.

\section{Necessary and sufficient graph properties \label{sec:cond}}

Assumption~\ref{ass1} is a necessary condition for the emergence of oscillations at an indecision breaking bifurcation. The following proposition singles out classes of communication and belief system graphs for which the leading eigenvalues of~\eqref{eq:jac} are necessarily real and therefore no oscillation in beliefs can emerge at the breaking of indecision.

\begin{proposition}[Graphs that never support oscillations]\label{prop: no oscil graphs}
Consider \eqref{EQ:value_dynamics} with communication graph $\mathcal{G}_a$ and belief system graph $\mathcal{G}_o$ with signed adjacency matrices $A_a,A_o$. Suppose at least one of the following statements is true: 1) $\mathcal{G}_a$ and $\mathcal{G}_o$ are undirected; 2) for both $\mathcal{G}_a$ and $\mathcal{G}_o$ 
there exist switching matrices $M_a$,$M_o$ 
such that $M_a A_a M_a$ and $M_o A_o M_o$ are eventually positive. Then the indecision-breaking bifurcation of the origin at $u = u^*$ cannot be a Hopf bifurcation.

\end{proposition}
\begin{proof}
1) All eigenvalues $\lambda \in \sigma(A_a)$ and $\mu \in \sigma(A_o)$ are real because $A_a,A_i$ are symmetric; all eigenvalues $\eta(u,\lambda,\mu)$ are also real which violates Assumption \ref{ass1}. 2) Eventually positive matrices have the strong Perron-Frobenius property by Proposition \ref{prop:PerFr}. Conjugation by a switching matrix preserves eigenvalues as it is a similarity transformation. Thus, $A_a$ and $A_o$ possess unique dominant real eigenvalues $\lambda = \rho(A_a)$ and $\mu = \rho(A_o)$. The eigenvalue $\eta(u,\lambda,\mu)$ is therefore a unique dominant eigenvalue of the Jacobian \eqref{eq:jac}, which violates Assumption \ref{ass1}.
\end{proof}

Proposition~\ref{prop: no oscil graphs} singles out two classes of communications and belief system graphs for which no oscillations in beliefs are possible at an indecision-breaking bifurcation. In the first class, both graphs are undirected. In the second class both graphs are {\it eventually structurally balanced}, that is, their adjacency matrices are switching equivalent to eventually positive matrices. The two classes have in common that all the loops between any pairs of agents are positive, which makes oscillations impossible.
Conversely, when some negative feedback loops are present in either the communication or the belief system graph, then leading complex eigenvalues can appear in either graph and oscillations are possible. 

\begin{proposition}[Graphs that support oscillations]\label{prop: oscil graphs}

Consider \eqref{EQ:value_dynamics} with communication graph $\mathcal{G}_a$ and belief system graph $\mathcal{G}_o$ with signed adjacency matrices $A_a,A_o$. Suppose there exists a switching matrix $M$ such that $M A_a M$ ($M A_o M $) is eventually positive, and suppose the leading eigenvalues of $A_o (A_a)$ are a complex-conjugate pair with positive real part. Then there exists a critical value $\gamma^*$ ($\beta^*$) such that whenever $\gamma > \gamma^*$ $(\beta > \beta^*)$, the indecision-breaking bifurcation of the origin at $u = u^*$ is a Hopf bifurcation. 
\end{proposition}
\begin{proof}
Consider without loss of generality an eventually structurally balanced $A_a$ with dominant eigenvalue $\lambda > 0$, and $A_o$ with leading eigenvalues $\mu, \overline{\mu}$. The eigenvalues $\eta(u,\lambda,\mu),\eta(u,\lambda,\overline{\mu})$ are the leading eigenvalues of the Jacobian \eqref{eq:jac} whenever $
\gamma \operatorname{Re}( \lambda - \lambda_a ) + \beta \operatorname{Re}(\mu - \mu_o) + \delta\operatorname{Re}(\mu \lambda - \mu_o \lambda_a) > 0$
for all $\lambda_a \in \sigma(A_a), \mu_o \in \sigma(A_o)$ with $\lambda_a \neq \lambda,\overline{\lambda}$; $\mu_a \neq \mu, \overline{\mu}$. This is satisfied whenever $\gamma > \max_{\lambda_a \in \sigma(A_a), \mu_o \in \sigma(A_o)} -\big(\beta \operatorname{Re}(\mu - \mu_o) + \delta\operatorname{Re}(\mu \lambda - \mu_o \lambda_a)\big)/(\lambda - \lambda_a) =:  \gamma^*$. Furthermore $K = \alpha + \gamma \lambda + \beta \operatorname{Re}(\mu) + \delta \lambda \operatorname{Re}(\mu) > 0$ and the necessary and sufficient conditions of Theorem \ref{thm:hopf} are satisfied. Analogous arguments establish existence of $\beta^*$. \end{proof}

In the classes of communication and belief system graphs singled out by Proposition~\ref{prop: oscil graphs}, the graph whose leading eigenvalues are not complex must be eventually structurally balanced. This ensures that the complex leading eigenvalues of  the communication or belief system graph (as appropriate) are mapped to complex leading eigenvalues of~\eqref{eq:jac}.

\section{Numerical examples \label{sec:ex}}

We explore how different communication and belief system graphs shape the emerging oscillations. In all the examples, $S_1(\cdot) = \tanh(\cdot)$, $S_2(\cdot) = \frac{1}{2}\tanh(2 \cdot)$.

\subsection{Single topic: communication-induced oscillations}

We first consider  \eqref{EQ:value_dynamics} in the case that agents evaluate a single topic so $\Zz_i = z_{i1} \in \mathds{R}$ and only the communication graph $\mathcal{G}_a$ plays a role in the belief dynamics. We denote  $z_{i1}$ by $z_i$ and the dynamics are 
\begin{equation}
    \dot{z}_i = - d \ z_i + u \  S_{1}\left( \alpha z_{i} + \gamma \textstyle \sum_{\substack{k=1 \\ k \neq i}}^\Na(A_a)_{ik} z_{k}\right). \label{eq:scalar-dyn}
\end{equation}
Since the belief system adjacency matrix $A_o=1$ in the one-topic  case, $\mathbf{v}_o=\mathbf{w}_o=\mu^{\dagger}=1$ in Theorem \ref{thm:hopf}. 

Consider \eqref{eq:scalar-dyn} for seven agents and communication graph $\mathcal{G}_a$ of Fig. \ref{fig:ex1}a, with model parameters $d = 1$, $\alpha = \gamma = 0.1$. The adjacency matrix $A_a$ has complex conjugate leading eigenvalues $\lambda_{\pm} = 0.90 \pm 0.43 i$, which generate the leading eigenvalues of the Jacobian $\eta_{\pm} = -d + u (\alpha +\gamma \lambda_{\pm})$. The conditions for oscillations in Theorem \ref{thm:hopf} are satisfied.
The origin loses stability at $u^* = 5.26$ and $b \approx - 0.0041$ from \eqref{eq:Hopf_stab_cond}, which predicts a supercritical bifurcation of stable oscillations. Since $|(\mathbf{v}_a)_i| = |(\mathbf{v}_a)_k|$ for $i,k = 1, \dots, 7$, all agent opinions oscillate with the same amplitude. By Theorem~\ref{thm:hopf} the predicted period of oscillation is approximately $27.53$, and the predicted oscillation phases of the seven agents relative to agent 1 are $(0, 3.59 , 0.90 , 4.49, 1.80 , 5.39, 2.69)$. Fig. \ref{fig:ex1}b illustrates these predictions.

\begin{figure}
    \centering
    \includegraphics[width=\linewidth]{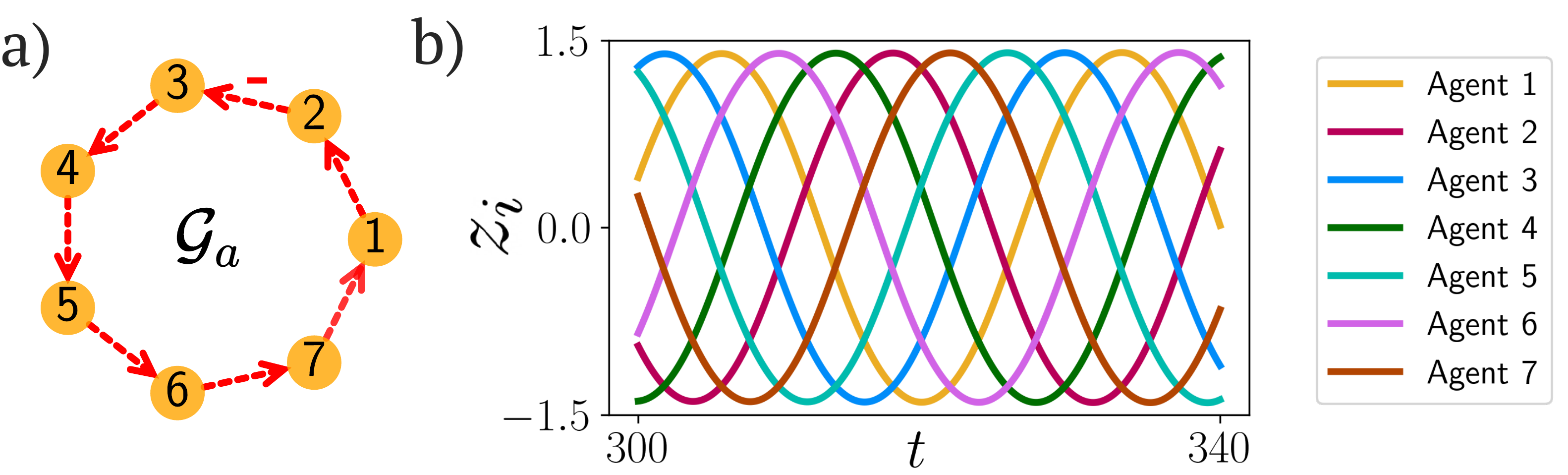}
    \caption{a) Communication graph for seven agents; red edges represent negative connections; b) trajectories of \eqref{eq:scalar-dyn} with the communication graph of a) from random initial conditions. Parameters: $d = 1$, $\alpha = \gamma = 0.1$, $u = 5.35$}
    \label{fig:ex1}
\end{figure}

\subsection{Multiple topics: communication and belief system-induced oscillations}

In the multiple-topic case, it is the \textit{interaction} of communication and belief system graphs that shapes the oscillations.
We stress that in all examples interchanging agents for topics and communication graph for belief system graph preserves the model dynamical behavior (modulo a reordering of state variables) but changes its interpretation: a given phase difference and amplitude oscillation pattern between the agents (topics) is mapped to the same phase difference and amplitude oscillation pattern between the topics (agents). We consider three examples with the same belief system (Fig. \ref{fig:ex2graphs}a). The adjacency matrix $A_o$ has a complex conjugate pair of leading eigenvalues, $\mu_{\pm} \approx 0.66 \pm 0.56 i$. For all three examples, $d = 1$, $\alpha = \gamma = 0.1$, and $\beta = \delta = 0.25$.

\begin{figure}
    \centering
    \includegraphics[width=\linewidth]{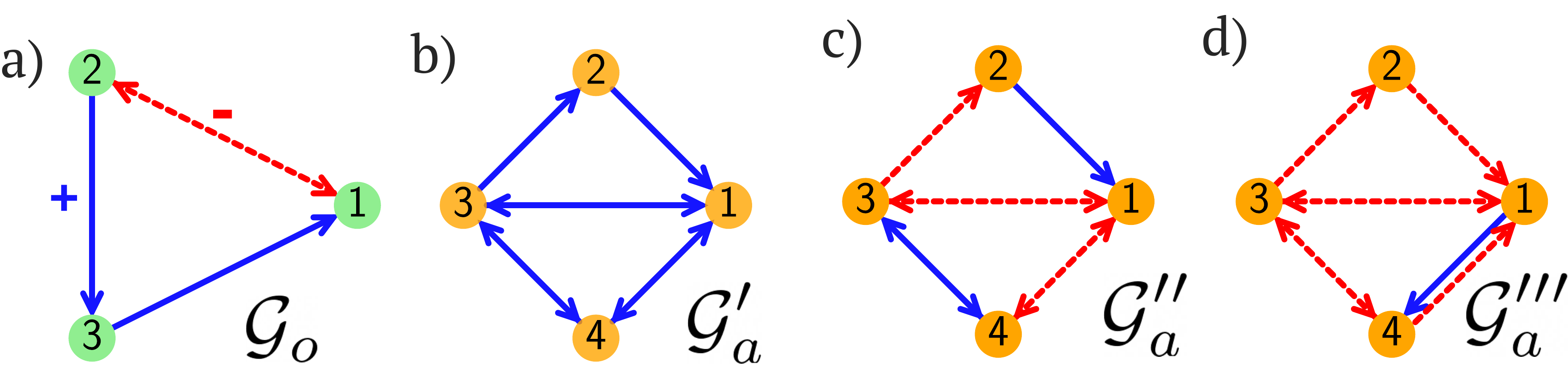}
    \caption{a) Belief system graph for three topics; b)-d) three different communication graphs for four agents with the same connectivity. Blue (red) edges represent positive (negative) connections.}
    \label{fig:ex2graphs}
\end{figure}

\subsubsection{Agreement oscillations}
Consider the strongly connected communication graph  $\mathcal{G}_a'$ (Fig.  \ref{fig:ex2graphs}b) with purely cooperative agents. Its adjacency matrix $A_a'$ has the strong Perron-Frobenius property, and the
conditions of Proposition \ref{prop: oscil graphs} are satisfied. By \eqref{eq:Hopf_stab_cond},  $b \approx -0.46 < 0$, and thus a supercritical bifurcation of stable periodic orbits is expected at $u^* = 0.93$ (Fig. \ref{fig:ex2a}). Phase differences between any two belief trajectories on the same topic must be zero (Fig. \ref{fig:ex2a}a) because $A_a$ has the strong Perron-Frobenius property and its dominant eigenvector $\mathbf{v}_a \succ 0$.
In contrast, each agent's beliefs on different topics are not in phase (Fig. \ref{fig:ex2a}b) as predicted by the entries of $\mathbf{v}_o$.

\begin{figure}
    \centering
    \includegraphics[width=\linewidth]{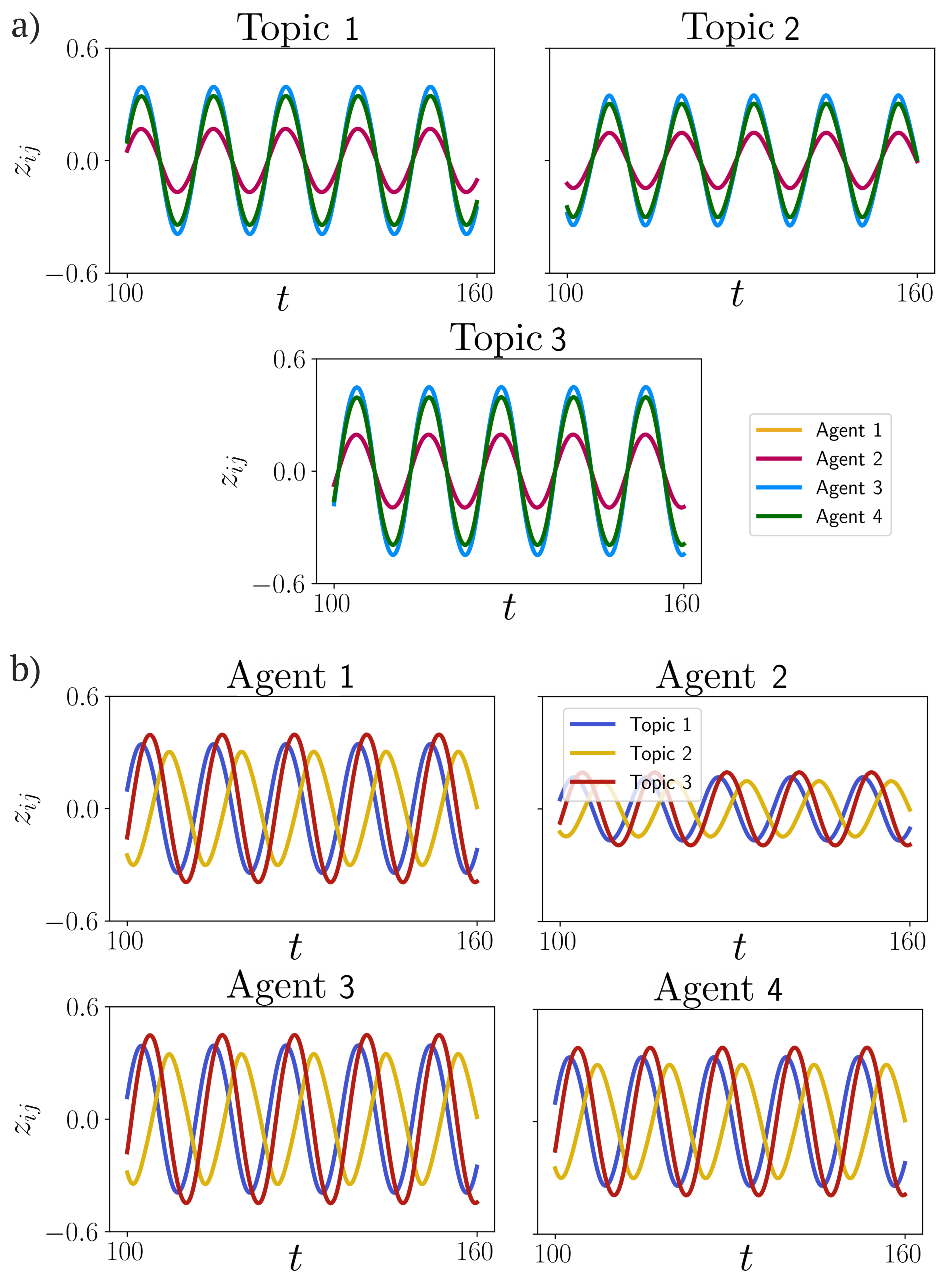}
    \caption{ Trajectories $z_{ij}(t)$ of \eqref{EQ:value_dynamics} with belief system $\mathcal{G}_o$ and communication graph $\mathcal{G}_a'$ of Fig.~\ref{fig:ex2graphs} from random initial conditions. Trajectories are grouped by a) topic  b) agent. In a) trajectories of agents 1 and 4 overlap in all three  plots.  Parameters: $d = 1$, $\alpha = \gamma = 0.1$, $\beta = \delta = 0.25$, $u = 1.25$ }
    \label{fig:ex2a}
\end{figure}

\subsubsection{Clustered disagreement oscillations} 

\begin{figure}
    \centering
    \includegraphics[width=\linewidth]{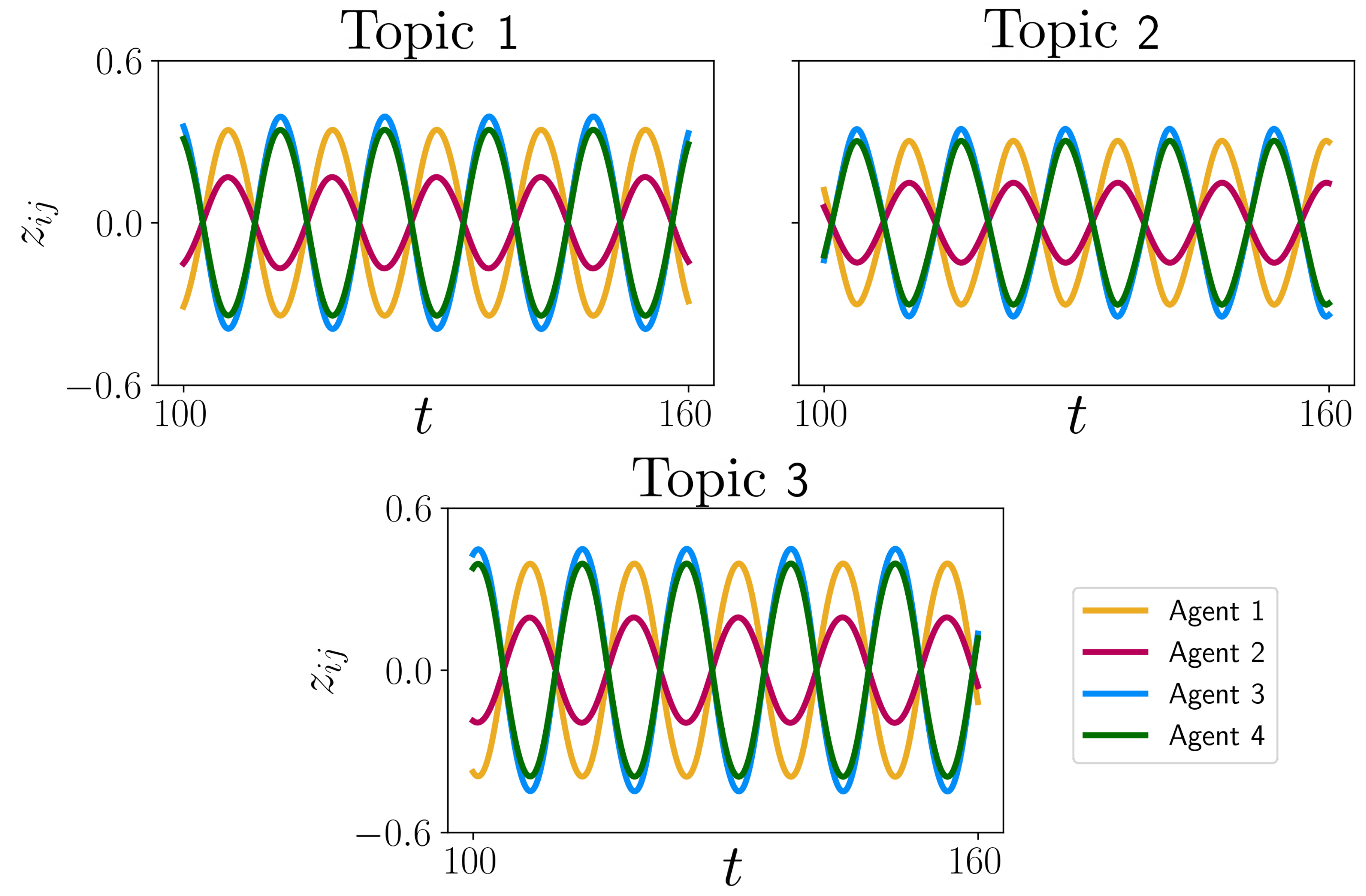}
    \caption{ Trajectories $z_{ij}(t)$ of \eqref{EQ:value_dynamics} with belief system $\mathcal{G}_o$ and communication graph $\mathcal{G}_a''$ of Fig.~\ref{fig:ex2graphs} from random initial conditions, grouped by topic. Parameters: $u = 1.25$, $d = 1$, $\alpha = \gamma = 0.1$, $\beta = \delta = 0.25$, $u = 1.25$}
    \label{fig:ex2b}
\end{figure}
Consider the mixed-sign communication graph, $\mathcal{G}_a''$ in Fig.~\ref{fig:ex2graphs}c. The adjacency matrix $A_a''$ of this graph is generated from $A_a'$ of the previous example as $A_a'' = M A_a' M$ where $M = \operatorname{diag}(1,1,-1,-1)$ is a switching matrix. As in the previous example, $b \approx -0.46$ and $u^* \approx 0.93$. In contrast to the previous example,
$\operatorname{sign}(\mathbf{v}_a)_1=\operatorname{sign}(\mathbf{v}_a)_2=-\operatorname{sign}(\mathbf{v}_a)_3=-\operatorname{sign}(\mathbf{v}_a)_4$. As a result, the beliefs of agents 1 and 2  oscillate in anti-phase with respect to the beliefs of agents 3 and 4 (Fig. \ref{fig:ex2b}).

\subsubsection{Asynchronous disagreement oscillations}
Consider the mixed-sign communication graph $\mathcal{G}_a'''$ in Fig. \ref{fig:ex2graphs}d, whose adjacency matrix $A_a'''$ has a complex-conjugate set of leading eigenvalues, $\lambda_{\pm} \approx 0.88 \pm 0.74 i$. The two pairs $(\lambda_{+},\mu_-)$, $(\lambda_-,\mu_+)$ generate the two complex-conjugate leading eigenvalues of $J(\mathbf{0},u)$ which satisfy the conditions of Theorem \ref{thm:hopf}. We compute $b \approx -0.13$ and a supercritical bifurcation of stable periodic orbits is expected at $u^* \approx 1.64$ (Fig. \ref{fig:ex2c}). In contrast to the previous examples, the leading eigenvectors of $J(\mathbf{0},u)$ are a product of two complex eigenvectors, and there is no phase synchronization in the resulting oscillations along any topic or within the agents' internal dynamics.

\begin{figure}
    \centering
    \includegraphics[width=\linewidth]{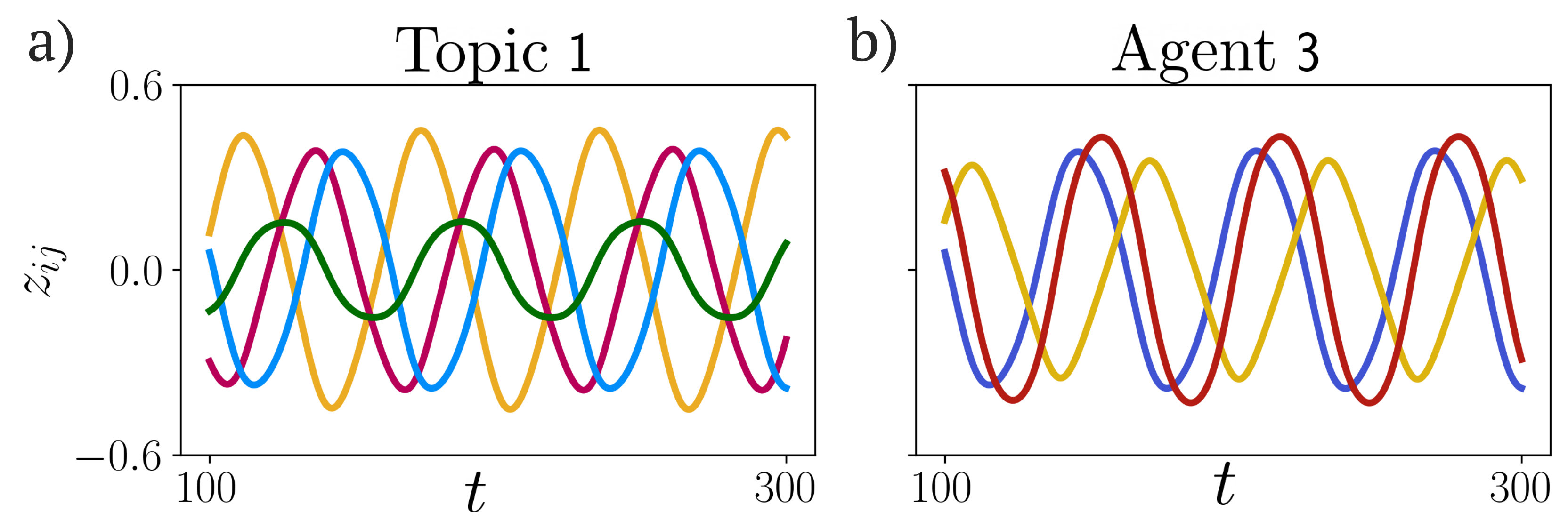}
    \caption{Representative trajectories $z_{ij}(t)$ of \eqref{EQ:value_dynamics} with belief system $\mathcal{G}_o$ and communication graph $\mathcal{G}_a'''$ of Fig.~\ref{fig:ex2graphs}, from random initial conditions. a) Beliefs of all agents on topic 1; b) beliefs of agent 1 on all topics. Color legend for a) and b) the same as in Fig. \ref{fig:ex2a} a) and b). Parameters: $d = 1$, $\alpha = \gamma = 0.1$, $\beta = \delta = 0.25$, $u = 1.7$}
    \label{fig:ex2c}
\end{figure}

\section*{Appendix: proof of Theorem \ref{thm:hopf}}
1) To establish existence of periodic orbits we check that the system \eqref{EQ:value_dynamics} under the stated assumptions satisfies the conditions of the Hopf bifurcation theorem \cite[Theorem 3.4.2]{guckenheimer2013nonlinear}. When $u = u^* = d/K$, the leading eigenvalues of \eqref{eq:jac} are a simple purely imaginary pair
$
\eta_{\pm}(u^*) = \pm i u^* \big| \gamma \lambda_c + \beta \mu_c + \delta ( \lambda_a \mu_c + \lambda_c \mu_o) \big| \neq 0 
$, 
which satisfies the eigenvalue condition (H1) of the Hopf theorem. Next, we check that the leading eigenvalues cross the imaginary axis with nonzero speed as $u$ is varied, i.e. 
$
    \frac{d}{du}\operatorname{Re}(\eta_{\pm}(u)) =  K > 0 
$, 
which satisfies the nonzero crossing speed condition (H2) of the Hopf theorem. Existence of periodic orbits directly follows by the Hopf theorem. By this theorem and by the definition of a center manifold \cite[Theorem 3.2.1]{guckenheimer2013nonlinear}, the solutions appear along a unique $W^s$ which is tangent at $u = u^*$ to $\mathcal{N}\big(J(\mathbf{0},u^*)\big) = \operatorname{span}\{ \operatorname{Re}(\mathbf{v}_a \otimes \mathbf{v}_o), \operatorname{Im}(\mathbf{v}_a \otimes \mathbf{v}_o) \} $.

To show 2) and 3) we first compute the coefficients of a third-order approximation of \eqref{EQ:value_dynamics} following the Lyapunov-Schmidt reduction for a Hopf bifurcation  \cite[Chapter VIII, Proposition 3.3]{Golubitsky1985}. This approximation reads
$
    f(y,u) = K y (u - u^*)  + \frac{1}{16} u^* b y^3
$, 
where $K$ is defined in \eqref{eq:K-osc} and $b$ is defined in \eqref{eq:Hopf_stab_cond}.
As long as $b \neq 0$, by \cite[Chapter VIII, Theorems 2.1 and 3.2]{Golubitsky1985} the reduced bifurcation equation $f(y,u)$ possesses a pitchfork bifurcation which is supercritical for $b<0$ and subcritical for $b>0$. When
$\lvert u - u^*\rvert$ is small, solutions to $f(y,u) = 0$ are in one-to-one correspondence with orbits of small amplitude periodic solutions to the system \eqref{EQ:value_dynamics} with period near $ 2 \pi/(u^*|\gamma \lambda_c + \beta \mu_c + \delta (\lambda_a \mu_c + \lambda_c \mu_o)|)=:1/\omega$. For $u$ near $u^*$, the small amplitude oscillations can be approximated to first order as scalar multiples of $e^{i \omega t} \mathbf{v}_a \otimes \mathbf{v}_o$ from which the conclusions on phase and amplitude difference between agents follow. When $b < 0$ ($>0$), the bifurcating periodic solutions are stable (unstable) by \cite[Chapter VIII, Theorem 4.1]{Golubitsky1985}. 



\bibliographystyle{./bibliography/IEEEtran}
\bibliography{./bibliography/references}

\begin{thebibliography}{10}
\providecommand{\url}[1]{#1}
\csname url@samestyle\endcsname
\providecommand{\newblock}{\relax}
\providecommand{\bibinfo}[2]{#2}
\providecommand{\BIBentrySTDinterwordspacing}{\spaceskip=0pt\relax}
\providecommand{\BIBentryALTinterwordstretchfactor}{4}
\providecommand{\BIBentryALTinterwordspacing}{\spaceskip=\fontdimen2\font plus
\BIBentryALTinterwordstretchfactor\fontdimen3\font minus
  \fontdimen4\font\relax}
\providecommand{\BIBforeignlanguage}[2]{{%
\expandafter\ifx\csname l@#1\endcsname\relax
\typeout{** WARNING: IEEEtran.bst: No hyphenation pattern has been}%
\typeout{** loaded for the language `#1'. Using the pattern for}%
\typeout{** the default language instead.}%
\else
\language=\csname l@#1\endcsname
\fi
#2}}
\providecommand{\BIBdecl}{\relax}
\BIBdecl

\bibitem{fink2002oscillation}
E.~L. Fink, S.~A. Kaplowitz, and S.~M. Hubbard, ``Oscillation in beliefs and
  decisions,'' \emph{The Persuasion Handbook: Developments in Theory and
  Practice. Thousand Oaks, CA: Sage Publications}, pp. 17--38, 2002.

\bibitem{stimson2018public}
J.~Stimson, \emph{Public opinion in America: Moods, cycles, and swings}.\hskip
  1em plus 0.5em minus 0.4em\relax Routledge, 2018.

\bibitem{Liu2014event}
S.-C. Liu, T.~Delbruck, G.~Indiveri, A.~Whatley, and R.~Douglas,
  \emph{Event-based neuromorphic systems}.\hskip 1em plus 0.5em minus
  0.4em\relax John Wiley \& Sons, 2014.

\bibitem{DeGroot1974}
M.~H. DeGroot, ``Reaching a consensus,'' \emph{Journal of the American
  Statistical Association}, vol.~69, no. 345, pp. 121--132, 1974.

\bibitem{OlfatiSaber2004}
R.~{Olfati-Saber} and R.~M. {Murray}, ``Consensus problems in networks of
  agents with switching topology and time-delays,'' \emph{IEEE Trans. Autom.
  Control}, vol.~49, no.~9, pp. 1520--1533, 2004.

\bibitem{friedkin2016network}
N.~E. Friedkin, A.~V. Proskurnikov, R.~Tempo, and S.~E. Parsegov, ``Network
  science on belief system dynamics under logic constraints,'' \emph{Science},
  vol. 354, no. 6310, pp. 321--326, 2016.

\bibitem{parsegov2016novel}
S.~E. Parsegov, A.~V. Proskurnikov, R.~Tempo, and N.~E. Friedkin, ``Novel
  multidimensional models of opinion dynamics in social networks,'' \emph{IEEE
  Trans. Autom. Control}, vol.~62, no.~5, pp. 2270--2285, 2016.

\bibitem{ye2019consensus}
M.~Ye, J.~Liu, L.~Wang, B.~D. Anderson, and M.~Cao, ``Consensus and
  disagreement of heterogeneous belief systems in influence networks,''
  \emph{IEEE Trans. Autom. Control}, vol.~65, no.~11, pp. 4679--4694, 2019.

\bibitem{pan2018bipartite}
L.~Pan, H.~Shao, M.~Mesbahi, Y.~Xi, and D.~Li, ``Bipartite consensus on
  matrix-valued weighted networks,'' \emph{IEEE Transactions on Circuits and
  Systems II: Express Briefs}, vol.~66, no.~8, pp. 1441--1445, 2018.

\bibitem{ye2020continuous}
M.~Ye, M.~H. Trinh, Y.-H. Lim, B.~D. Anderson, and H.-S. Ahn, ``Continuous-time
  opinion dynamics on multiple interdependent topics,'' \emph{Automatica}, vol.
  115, p. 108884, 2020.

\bibitem{ahn2020opinion}
H.-S. Ahn, Q.~Van~Tran, M.~H. Trinh, M.~Ye, J.~Liu, and K.~L. Moore, ``Opinion
  dynamics with cross-coupling topics: Modeling and analysis,'' \emph{IEEE
  Trans. Computat. Social Syst.}, vol.~7, no.~3, pp. 632--647, 2020.

\bibitem{wang2022characterizing}
C.~Wang, L.~Pan, H.~Shao, D.~Li, and Y.~Xi, ``Characterizing bipartite
  consensus on signed matrix-weighted networks via balancing set,''
  \emph{Automatica}, vol. 141, p. 110237, 2022.

\bibitem{bizyaeva2023tac}
A.~Bizyaeva, A.~Franci, and N.~E. Leonard, ``Nonlinear opinion dynamics with
  tunable sensitivity,'' \emph{IEEE Trans. Autom. Control}, vol.~68, no.~3, pp.
  1415--1430, 2023.

\bibitem{FranciSIADS}
\BIBentryALTinterwordspacing
A.~Franci, M.~Golubitsky, I.~Stewart, A.~Bizyaeva, and N.~E. Leonard,
  ``Breaking indecision in multi-agent, multi-option dynamics,'' 2022.
  [Online]. Available: \url{https://arxiv.org/abs/2206.14893}
\BIBentrySTDinterwordspacing

\bibitem{bizyaeva2022switching}
A.~Bizyaeva, G.~Amorim, M.~Santos, A.~Franci, and N.~E. Leonard, ``Switching
  transformations for decentralized control of opinion patterns in signed
  networks: Application to dynamic task allocation,'' \emph{IEEE Control
  Systems Letters}, vol.~6, pp. 3463--3468, 2022.

\bibitem{noutsos2006perron}
D.~Noutsos, ``On {P}erron--{F}robenius property of matrices having some
  negative entries,'' \emph{Linear Algebra and its Applications}, vol. 412, no.
  2-3, pp. 132--153, 2006.

\bibitem{Golubitsky1985}
M.~Golubitsky and D.~G. Schaeffer, \emph{Singularities and Groups in
  Bifurcation Theory}, ser. Applied Mathematical Sciences.\hskip 1em plus 0.5em
  minus 0.4em\relax New York, NY: Springer-Verlag, 1985, vol.~51.

\bibitem{horn2012matrix}
R.~A. Horn and C.~R. Johnson, \emph{Matrix Analysis}.\hskip 1em plus 0.5em
  minus 0.4em\relax Cambridge University Press, 2012.

\bibitem{horn1991topics}
------, \emph{Topics in Matrix Analysis, 1991}.\hskip 1em plus 0.5em minus
  0.4em\relax Cambridge University Press, Cambridge, 1991.

\bibitem{Khalil2002}
H.~Khalil, \emph{Nonlinear Systems}, 3rd~ed.\hskip 1em plus 0.5em minus
  0.4em\relax Pearson Education International Inc., 2000.

\bibitem{guckenheimer2013nonlinear}
J.~Guckenheimer and P.~Holmes, \emph{Nonlinear Oscillations, Dynamical Systems,
  and Bifurcations of Vector Fields}.\hskip 1em plus 0.5em minus 0.4em\relax
  Springer, 2013.

\end{thebibliography}

\end{document}